\documentclass[reqno]{amsart}
\usepackage{amssymb,amscd,verbatim, amsthm,graphics, color,latexsym,amsmath,multicol}
\usepackage{fancybox}
\usepackage{longtable}

\newcommand \fk[1]{{{\mathfrak #1}}}
\newcommand \C[1]{{\mathcal #1}}

\newcommand \wti[1]{{\widetilde {#1}}}

\newcommand\fg{\mathfrak g}

\newcommand \bC{{\mathbb C}}

\newcommand \bH{{\mathbb H}}
\newcommand \bR{{\mathbb R}}
\newcommand \bZ{{\mathbb Z}}

\newcommand \bQ{{\mathbb Q}}

\newcommand\eg{{\it e.g.~ }}

\newcommand\ep{{\epsilon}}

\newcommand\om{{\omega}}
\newcommand\al{{\alpha}}

\newcommand\fh{{\mathfrak h}}

\newtheorem{theorem}{Theorem}[subsection]
\newtheorem{conjecture}{Conjecture}[subsection]
\newtheorem{corollary}{Corollary}[subsection]
\newtheorem{lemma}{Lemma}[subsection]
\newtheorem{proposition}{Proposition}[subsection]
\newtheorem{definition}{Definition}[subsection]

\newtheorem{example}{Example}[subsection]

\newcommand\Hom{\operatorname{Hom}}

\newcommand\Ima{\operatorname{Im}}

\newcommand\re{\operatorname{Re}}

\newcommand\triv{\mathsf{triv}}
\newcommand\sgn{\mathsf{sign}}

\newcommand\St{\mathsf{St}}

\newlength{\tabwidth}
\newlength{\tabheight}
\newlength{\tabrule}
\newlength{\tabwidthx}
\newlength{\tabheightx}

\def\gentabbox#1#2#3#4{\vbox to \tabheight{\setlength{\tabrule}{#3}%
  \setlength{\tabwidthx}{#1\tabwidth}\addtolength{\tabwidthx}{\tabrule}%

\setlength{\tabheightx}{#2\tabheight}\addtolength{\tabheightx}{-\tabheight}%
  \hbox to #1\tabwidth{%
    \hspace{-0.5\tabrule}\rule{\tabrule}{#2\tabheight}\hspace{-\tabrule}%
    \vbox to #2\tabheight{\hsize=\tabwidthx%
      \vspace{-0.5\tabrule}\hrule width\tabwidthx height\tabrule%
      \vspace{-0.5\tabrule}\vfil%
      \hbox to \tabwidthx{\hss#4\hss}%
        \vfil\vspace{-0.5\tabrule}%
      \hrule width\tabwidthx height\tabrule\vspace{-0.5\tabrule}}%
    \hspace{-\tabrule}\rule{\tabrule}{#2\tabheight}\hspace{-0.5\tabrule}}%
  \vspace{-\tabheightx}}}
\def\genblankbox#1#2{\vbox to \tabheight{\vfil\hbox to
#1\tabwidth{\hfil}}}
\def\tabbox#1#2#3{\gentabbox{#1}{#2}{0.4pt}{\strut #3}}

\catcode`\:=13 \catcode`\.=13 \catcode`\;=13 
\catcode`\>=13 \catcode`\^=13
\def:#1\\{\hbox{$#1$}}
\def.#1{\tabbox{1}{1}{$#1$}}
\def>#1{\tabbox{2}{1}{$#1$}}
\def^#1{\tabbox{1}{2}{$#1$}}
\def;{\genblankbox{1}{1}\relax}
\catcode`\:=12 \catcode`\.=12 \catcode`\;=12 
\catcode`\>=12 \catcode`\^=7

\newenvironment{tableau}{\bgroup\catcode`\:=13 \catcode`\.=13
  \catcode`\;=13 \catcode`\>=13 \catcode`\^=13
  \setlength{\tabheight}{3ex}\setlength{\tabwidth}{3ex}%
  \def\b##1##2##3{\gentabbox{##1}{##2}{1.2pt}{\vbox{##3}}}%
  \def\n##1##2##3{\gentabbox{##1}{##2}{0.4pt}{\vbox{##3}}}%
  \vbox\bgroup\offinterlineskip}{\egroup\egroup}

\numberwithin{equation}{subsection}

\begin{document}

\title[Unitary Hecke modules with nonzero Dirac
cohomology]{Unitary Hecke algebra modules with nonzero Dirac
  cohomology}

\author{Dan Barbasch}
       \address[D. Barbasch]{Dept. of Mathematics\\
               Cornell University\\Ithaca, NY 14850}
       \email{barbasch@math.cornell.edu}
\thanks{\noindent The first author was partially supported by NSF grants DMS-0967386, DMS-0901104 and an NSA-AMS grant. The second author was partially supported by NSF DMS-0968065 and NSA-AMS 081022}
\author{Dan Ciubotaru}
        \address[D. Ciubotaru]{Dept. of Mathematics\\ University of
          Utah\\ Salt Lake City, UT 84112}
        \email{ciubo@math.utah.edu}


\begin{abstract}In this paper, we review the construction of the Dirac
  operator for graded affine Hecke algebras and calculate the Dirac
  cohomology of irreducible unitary modules for the graded Hecke algebra of $gl(n)$.
\end{abstract}

\dedicatory{To Nolan Wallach with admiration}

\maketitle

\section{Introduction}\label{sec:1}
The Dirac operator plays an important role in the representation
theory of real reductive Lie groups. An account of the definition,
properties and some applications can be found in \cite{BW}. It is well
known, starting with the work of \cite{AS} and \cite{P}, that discrete
series occur in the kernel of the Dirac operator. Work of Enright and Wallach
\cite{EW} generalizes these results to other types of
representations. Other uses are to provide, via the \textit{Dirac
  inequality}, {introduced by Parthasarathy},  necessary
conditions for unitarity. One of the most 
striking applications is that for regular integral infinitesimal
character, the Dirac inequality gives precisely the unitary dual, and
determines the unitary representations with nontrivial $(\fk
g,K)-$cohomology.  

Given these properties, Vogan has introduced the notion of Dirac
cohomology; this was studied extensively in \cite{HP} and subsequent
work. One can argue that Dirac cohomology is a generalization of $(\fk
g,K)-$cohomology. While a representation has nontrivial  $(\fk
g,K)-$cohomology only if its infinitesimal character is regular
integral, the corresponding condition necessary for Dirac cohomology
to be nonzero is more general; certain representations with singular
and nonintegral infinitesimal character will also have nontrivial
Dirac cohomology. 
 
\bigskip
In this paper, we prove new results about an analogue of the Dirac
operator in the case of the graded affine Hecke algebra, introduced in \cite{BCT}. This operator can be thought of as the analogue of the Dirac operator in the case of a $p$-adic group. 
One of our results is to determine the behaviour of the Dirac
cohomology with respect to Harish-Chandra type induction. In the real
case, a  unitary representation has nontrivial $(\fk g
,K)-$cohomology, if and only if it is (essentially) obtained from the
trivial representation on a Levi component via the derived functor
construction. For unitary representations with
nontrivial Dirac cohomology the infinitesimal character can be nonintegral and singular. So conjecture instead,  that unitary representations with nontrivial Dirac cohomology  are
all cohomologically induced from unipotent (in the sense of \cite{A}) representations. To
investigate this conjecture we explore the Dirac cohomology of
unipotent representations for graded affine Hecke algebras. In particular, we compute part of the
cohomology of spherical unipotent representations for affine Hecke algebras of all types. In the case of type
$A$ we go further; we compute the cohomology of all unitary modules.  

\smallskip

This paper was written while we were guests of the Max Planck Institute in Bonn as part of the program \textit{Analysis on Lie groups}. We would like to thank the institute for its hospitality, and the organizers for making the program  possible, and providing the environment to do this research.

\section{Dirac cohomology for graded Hecke algebras}\label{sec:2}

In this section we review the construction and properties of the Dirac
operator from \cite{BCT} and the classification of spin projective Weyl group representations from \cite{C}.

\subsection{Root systems} We fix an $\bR$-root system
$\Phi=(V,R,V^\vee, R^\vee)$: $V, V^\vee$ are finite dimensional $\bR$-vector spaces, with a perfect bilinear pairing $(~,~): V\times V^\vee\to \bR$, so that $R\subset V\setminus\{0\},$  $R^\vee\subset V^\vee\setminus\{0\}$ are finite subsets in bijection
\begin{equation}
R\longleftrightarrow R^\vee,\ \al\longleftrightarrow\al^\vee,\ \text{such that }(\al,\al^\vee)=2.
\end{equation}
The reflections
\begin{equation}
s_\al: V\to V,\ s_\al(v)=v-(v,\al^\vee)\al, \quad s_\al:V^\vee\to V^\vee,\ s_\al(v')=v'-(\al,v')\al^\vee, \quad \al\in R,
\end{equation}
leave $R$ and $R^\vee$ invariant, respectively. Let $W$ be the subgroup of $GL(V)$ (respectively $GL(V^\vee)$) generated by $\{s_\al:~\al\in R\}$.

We will assume that the root system $\Phi$ is reduced and crystallographic.
We will fix a choice of simple roots $\Pi\subset R$, and consequently, positive roots $R^+$ and positive coroots $R^{\vee,+}.$ Often, we will write $\al>0$ or $\al<0$ in place of $\al\in R^+$ or $\al\in (-R^+)$, respectively. 

\smallskip

We fix a $W$-invariant inner product $\langle~,~\rangle$ on
$V$.  Denote also by $\langle~,~\rangle$ the dual inner product on  $V^\vee.$ If $v$ is a vector in $V$ or $V^\vee$, we denote $|v|:=\langle v,v\rangle^{1/2}.$

\subsection{The Clifford algebra} A classical reference for the Clifford algebra is \cite{Ch} (see also
section II.6 in \cite{BW}). Denote by $ C(V)$ the Clifford algebra defined by $V$
and the inner product $\langle~,~\rangle$.  More precisely, $
C(V)$ is the
quotient of the tensor algebra of $V$ by the ideal generated by
$$\om\otimes \om'+\om'\otimes \om+2\langle \om,\om'\rangle,\quad
\om,\om'\in V.$$
Equivalently,  $ C(V)$ is the associative algebra
with unit generated by $V$ with relations:
\begin{equation}
\om\om'+\om'\om=-2\langle\om,\om'\rangle.
\end{equation}
Let $\mathsf{O}(V)$ denote the group of orthogonal transformation of
$V$ with respect to $\langle~,~\rangle$. This acts by algebra
automorphisms on $ C(V)$, and the action of $-1\in
\mathsf{O}(V)$ induces a grading
\begin{equation}
 C(V)= C(V)_{\mathsf{even}}+  C(V)_{\mathsf{odd}}.
\end{equation}
Let $\ep$ be the automorphism of $ C(V)$ which is $+1$ on $
C(V)_{\mathsf{even}}$ and $-1$ on $ C(V)_{\mathsf{odd}}$.
Let ${}^t$ be the transpose {anti}automorphism of $ C(V)$
characterized by
\begin{equation}
\om^t=-\om,\ \om\in V,\quad (ab)^t=b^ta^t,\ a,b\in C(V).
\end{equation}
The Pin group  is
\begin{equation}\label{pin}
\mathsf{Pin}(V)=\{a\in  C(V):~ \ep(a) V a^{-1}\subset
V,~ a^t=a^{-1}\}.
\end{equation}
It sits in a short exact sequence
\begin{equation}\label{ses}
1\longrightarrow \bZ/2\bZ \longrightarrow
\mathsf{Pin}(V)\xrightarrow{\ \ p\ \ } \mathsf{O}(V)\longrightarrow 1,
\end{equation}
where the projection $p$ is given by $p(a)(\om)=\ep(a)\om a^{-1}$.

\medskip

If $\dim V$ is even, the Clifford algebra $C(V)$ has a unique (up to equivalence) complex simple module
$(\gamma, S)$ of dimension $2^{\dim
    V/2}$, endowed with a positive
definite Hermitian form $\langle ~,~\rangle_{ S}$ such that
\begin{equation}\label{eq:unitary}
\langle\gamma(a)s,s'\rangle_{ S}=\langle s,\gamma(a^t) s'\rangle_{ S},\quad\text{for all
  }a\in  C(V)\text{ and } s,s'\in  S. 
\end{equation}
When $\dim V$ is odd, there are two simple inequivalent {complex}
modules $(\gamma_+,S^+),$ $(\gamma_-,S^-)$ of dimension  $2^{[\dim
  V/2]}$. {Analogous to (\ref{eq:unitary}), these
  modules admit
an invariant positive definite Hermitian form.}
In order to simplify the formulation of the results, we will often
refer to any one of $S$, $S^+,$ $S^-$, as a spin module. 

Via (\ref{pin}), a spin module $S$ is an irreducible unitary
$\mathsf{Pin}(V)$ representation.

\subsection{The pin cover $\wti W$ of the Weyl group} The Weyl group $W$ acts by orthogonal
transformations on $V$, so one can embed $W$ as a subgroup of
$\mathsf{O}(V).$ We define the group $\wti W$ in
$\mathsf{Pin}(V)$: 
\begin{equation}
\wti W:=p^{-1}({W})\subset \mathsf{Pin}(V),\text{ where $p$
  is as in (\ref{ses}).}
\end{equation}
The group $\wti W$ has a Coxeter presentation similar to that of $W$.
Recall that as a
Coxeter group, $W$ has a presentation:
\begin{equation}
W=\langle s_{\al},~\al\in\Pi|\  (s_\al
  s_\beta)^{m(\al,\beta)}=1, ~\al,\beta\in\Pi\rangle,
\end{equation}
for certain positive integers $m(\al,\beta).$
Theorem 3.2 in \cite{Mo} exhibits $\wti W$ as 
\begin{equation}
\wti W=\langle z,\wti s_{\al},~\al\in\Pi|\  z^2=1,~(\wti s_\al
  \wti s_\beta)^{m(\al,\beta)}=z, ~\al,\beta\in\Pi\rangle.
\end{equation}

We call a representation $\wti\sigma$ of $\wti W$ genuine
(resp. non-genuine) if
$\wti\sigma(z)=-1$ (resp. $\wti\sigma(z)=1$). The non-genuine $\wti
W$-representations are the ones that factor through $W$.
We say that two genuine $\wti W$-types $\sigma_1,\sigma_2$ are associate if $\sigma_1\cong\sigma_2\otimes\sgn$.

Since $\wti W\subset\mathsf{Pin}(V)$, we can regard $S$ if $\dim V$ is even (resp. $S^\pm$ if $\dim V$ is odd) as unitary (genuine)
$\wti W$-representations. If $R$ spans $V$, they are irreducible representations (\cite[Theorem 3.3]{Mo}). When $\dim V$ is odd, $S^+$ and $S^-$ are associate, while if $\dim V$ is even, $S$ is self-associate.

\begin{definition}[{\cite[\S 3.4]{BCT}}] Define the Casimir element of $\wti W$:
\begin{equation}\label{omWtilde}
\Omega_{\wti W}=z\sum_{\substack{\al>0,\beta>0\\s_\al(\beta)<0}} 
 |\al^\vee| |\beta^\vee| ~\wti s_\al \wti s_\beta\in \mathbb C[\wti
 W]^{\wti W}.
\end{equation}
Every $\wti\sigma\in \widehat{\wti W}$ acts on
$\Omega_{\wti W}$ by a scalar, which we denote
$\wti\sigma(\Omega_{\wti W}).$
\end{definition}

Before stating Theorem \ref{t:intro}, we need to introduce more
notation. Assume that $R$ spans $V$ and let $\fg$ be the complex semisimple Lie algebra with root system $\Phi$ and Cartan subalgebra $\fh=V^\vee\otimes_\bR \bC$, and let
$G$ be the simply connected Lie group with Lie algebra $\fg$. Extend the inner product from $V^\vee$ to $\fh.$ Let us denote by $\C T(G)$ the set of
$G$-conjugacy classes of Jacobson-Morozov triples $(e,h,f)$ in $\fg$. We set:
\begin{equation}\label{eq:tzero}
\C T_0(G)=\{[(e,h,f)]\in \C T(G): \text{ the centralizer of }\{e,h,f\} \text{
  in }\fg\text{ is a toral subalgebra}\}.
\end{equation}
For every class in $\C T(G)$, we may (and will) choose a representative $(e,h,f)$ such that $h\in\fh.$
For every
nilpotent element $e$, let $A(e)$ denote the A-group in $G$, and let
$\widehat {A(e)}_0$ denote the set of representations of $A(e)$ of
Springer type. For every $\phi\in \widehat{A(e)}_0$, let
$\sigma_{(e,\phi)}$ be the associated Springer representation.
Normalize the Springer correspondence so that
$\sigma_{0,\text{triv}}=\sgn$.

\begin{theorem}[\cite{C}]\label{t:intro} There is a surjective map
\begin{equation}
\Psi:\widehat{\wti W}_{\mathsf{gen}} \longrightarrow \C T_0(G),
\end{equation}
with the following properties:
\begin{enumerate}
\item If $\Psi(\wti\sigma)=[(e,h,f)]$, then we have
\begin{equation}
\wti\sigma(\Omega_{\wti W})=\langle h,h\rangle,
\end{equation}
where $\Omega_{\wti W}$ is as in (\ref{omWtilde}).
\item   Let $(e,h,f)\in \C T_0(G)$ be given. For every  Springer representation $\sigma_{(e,\phi)}$,
  $\phi\in\widehat {A(e)}_0$, and every spin $\wti W$-module $S$,
  there exists $\wti \sigma\in \Psi^{-1}[(e,h,f)]$ such that $\wti\sigma$ appears with
  nonzero multiplicity in the tensor product
  $\sigma_{(e,\phi)}\otimes S$. Conversely, for every $\wti\sigma\in \Psi^{-1}[(e,h,f)]$, there exists a
  spin $\wti W$-module $S$ and a Springer representation
  $\sigma_{(e,\phi)}$, such that $\wti\sigma$ is contained in
  $\sigma_{(e,\phi)}\otimes S.$
\end{enumerate}
\end{theorem}
Since $\triv(\Omega_{\wti W})=\sgn(\Omega_{\wti W})$, 
Theorem \ref{t:intro}(1) says in particular that any two associate
genuine $\wti W$-types $\wti\sigma_1,\wti\sigma_2$ lie in the same fiber of $\Psi$.

\subsection{The graded Hecke algebra} 
Recall the real root system $\Phi=(V,R,V^\vee,R^\vee)$. The complexifications of $V,
V^\vee$ are denoted by $V_\bC,V_\bC^\vee$. We denote by $S(V_\bC)$ the symmetric
algebra in $V_\bC.$

\begin{definition}[\cite{L}]The graded affine Hecke algebra $\bH$ (with
  equal parameters) is defined as follows:
\begin{enumerate}
\item[(i)] as a $\bC$-vector space, it is $S(V_\bC)\otimes\bC[W]$;
\item[(ii)] $S(V_\bC)$ and $\bC[W]$ have the usual algebra structures as
  subalgebras;
\item[(iii)] the cross relations are
$$s_\al\cdot\xi-s_\al(\xi)\cdot s_\al=\langle\al,\xi\rangle,$$
for every $\al\in \Pi$ and $\xi\in V_\bC.$
\end{enumerate}
\end{definition}

\begin{definition}\label{d:casimir} Let $\{\om_i:i=1,n\}$ and $\{\om^i: i=1,n\}$ be dual bases of $V$ with respect to $\langle~,~\rangle$. 
Define the Casimir element of $\bH$:
$\Omega=\sum_{i=1}^n\omega_i\omega^i\in \bH$.
\end{definition}
It is easy to see that the element $\Omega$ is independent of the
choice of bases and central in $\bH$. Moreover, if $(\pi,X)$ is an irreducible
  $\bH$-module with central character $\chi_\nu$ for $\nu \in V_\bC^\vee$,
then $\pi$ acts on $\Omega$ by the scalar $\langle \nu,\nu\rangle.$

The algebra $\bH$ has a natural conjugate linear
anti-involution defined on generators as follows:
\begin{equation} \label{eq:tomdef}
\begin{aligned}
&w^*={w^{-1}},\quad w\in W,\\
&\omega^*=-\omega+\sum_{\beta>0} (\beta,\omega) {s_\beta},\quad
{\omega\in V}.
\end{aligned}
\end{equation}

An $\bH$-module $(\pi,X)$ 
is said to be Hermitian 
if there exists a Hermitian form $(~,~)_X$ on $X$ which is
invariant in the sense that
$(\pi(h)x,y)_X={{(x,\pi(h^*)y)}_X},$ for all $h\in\bH,$ $ x,y\in X$.
If such a form exists which is also positive definite, then $X$ is
said to be unitary.

For every $\omega\in V$, define
\begin{equation}\label{omtilde}
\wti\om=\om-\frac 12 \sum_{\beta>0} (\beta,\omega) {s_\beta}
\; \in \; \bH.
\end{equation}
It is immediate that $\wti\omega^* = - \wti\omega$.

\begin{definition}[\cite{BCT}]
\label{d:dirac}
Let  $\{\om_i\}$, $\{\om^i\}$ be dual bases of $V$.
The Dirac element is defined as
\[
\C D = \sum_i \wti \omega_i \otimes \omega^i \in \bH \otimes {C(V)}.
\]
It is elementary to verify that $\C D$ does not depend on the choice of
dual bases.

We will usually work with
a fixed spin module $(\gamma, S)$ for {$C(V)$} 
and a fixed $\bH$-module $(\pi,X)$.  Define
the Dirac operator for $X$ (and $S$)
as $D=(\pi \otimes \gamma)(\C D)$.
\end{definition}

Suppose $X$ is a Hermitian $\bH$-module with invariant form $(~,~)_X$.
Endow $X \otimes S$ with the
Hermitian form $(x\otimes s, x' \otimes s')_{X \otimes S} 
= (x,x')_X  \langle s,s'\rangle_S$.  {Analogous to results of
  Parthasarathy in the real case,} the operator $D$ is self adjoint with
respect to $(~,~)_{X \otimes S}$,
\begin{equation}
( D (x\otimes s), x'\otimes s')_{X\otimes S}=
(x\otimes s,D(x'\otimes s'))_{X\otimes S}
\end{equation}
Thus a Hermitian $\bH$-module 
is unitary only if
\begin{equation}
\label{eq:dcriterion}
(D^2 (x\otimes s), x\otimes s)_{X\otimes S} \geq 0, \qquad \text{ for all $x\otimes s \in X \otimes S$}.
\end{equation}
We write $\Delta_{\wti W}$ for the diagonal embedding of $\bC[\wti W]$ into
{$\bH \otimes C(V)$} defined by
extending
$\Delta_{\wti W}(\wti w) = {p(\wti w)} \otimes \wti w$
linearly.

For $\wti w\in \wti W$, one can easily see that
\begin{equation}\label{Winv}
\Delta_{\wti W}(\wti w) \C D
=\sgn(\wti w) \C D \Delta_{\wti W}(\wti w)
\end{equation}
as elements of {$\bH \otimes C(V)$}. In particular, the kernel of the Dirac operator on $X \otimes S$ is 
invariant under $\wti W$.

\begin{theorem}[\cite{BCT}]\label{t:dirac} The square of the Dirac
  element equals
\begin{equation}
\C D^2=-\Omega\otimes 1+\frac 14\Delta_{\wti W}(\Omega_{\wti W}),
\end{equation}
in $\bH \otimes C(V)$.
\end{theorem}

\subsection{Dirac cohomology}
To have a uniform notation, we will denote a spin module by $S^\ep$. If $\dim V$ is even, then $S^\ep$ is $S$, the unique spin module, and if $\dim V$ is odd, then $\ep$ could be $+$ or $-$.
\begin{definition}
\label{d:dcoh}
In the setting of Definition \ref{d:dirac}, define
\begin{equation}
H^D_\ep(X):=\ker D\big / \left(\ker D\cap \Ima D\right )
\end{equation}
and call it the Dirac cohomology of $X$. (The symbol $\ep$ denotes the dependence on the spin module $S^\ep$.)
If $X$ is unitary, the self-adjointness of $D$ implies that
$\ker(D) \cap \Ima (D) = 0$, and so $H^D_\ep(X) = \ker (D)$.
\end{definition}

Vogan's conjecture takes the following form.

\begin{theorem}[{\cite[Theorem 4.8]{BCT}}]
\label{t:hpv}
Suppose $(\pi,X)$ is an $\bH$ module with central character $\chi_\nu$
with $\nu \in V_\bC^\vee$.
Suppose 
that $H^D_\ep(X) \neq 0$ and let $(\wti\sigma,\wti U)$ be an irreducible  
representation of $\wti W$ such that $\Hom_{\wti W}(\wti U, H^D_\ep(X) )
\neq 0$. If $\Psi(\wti\sigma)=[(e,h,f)]\in\C T_0(G)$, 
then $\nu=\frac 12h.$
\end{theorem}

As a corollary, we find the following formula for $H^D_\ep(X).$
\begin{corollary}
Assume $X$ is an $\bH$ module with central character $\chi_{\frac 12 h}$, for some $[(e,h,f)]\in\C T_0(G)$ (otherwise $H^D_\ep(X)=0$). Then, as a $\wti W$-module
\begin{equation}\label{HDdecomp}
H^D_\ep(X)=\sum_{\wti\sigma\in \Psi^{-1}(e,h,f)}\sum_{\mu\in\widehat W}[\wti\sigma:\mu\otimes S^\ep][X|_W:\mu]~\wti\sigma.
\end{equation}
\end{corollary}

Theorem \ref{t:hpv} has an easy weak converse, which will be useful in
applications.

\begin{proposition}\label{criterion}
Assume that $(\pi,X)$ is a unitary $\bH$ module with central character
$\chi_\nu$, $\nu\in V_\bC^\vee$ and that there exists an irreducible
$\wti W$-type $(\wti\sigma,\wti U)$ such that $\Hom_{\wti W}(\wti
U,X\otimes S^\ep)\neq 0$ and
$\langle\nu,\nu\rangle=\wti\sigma(\Omega_{\wti W}).$ Then $\Hom_{\wti
  W}(\wti U, H^D_\ep(X))\neq 0,$ and in particular $H^D_\ep(X)\neq 0.$ 
\end{proposition}

\begin{proof}
Let $x\otimes s$ be an element of $X\otimes S^\ep$ in the isotypic
component of $\wti\sigma$. Then $D^2(x\otimes
s)=-\langle\nu,\nu\rangle+\wti\sigma(\Omega_{\wti W})=0$. Since $X$ is
assumed unitary, the operator $D$ is self-adjoint on $X\otimes S$ and
thus $\ker D^2=\ker D.$ This implies $x\otimes s\in \ker D(=H_\ep^D(X).$
\end{proof}

\subsection{An induction lemma}\label{s:2.6}
Let $(V_M,R_M,V_M^\vee,R_M^\vee)$ be a root subsystem of\newline
$(V,R,V^\vee,R^\vee).$ Let $\Pi_M\subset \Pi$ be the corresponding
simple roots and $W_M\subset W$ the reflection subgroup. Let $\bH_M$
denote the Hecke subalgebra of $\bH$ given by this root subsystem.
Denote by $V_N$ the orthogonal complement of $V_M$
in $V$ with respect to the fixed product $\langle~,~\rangle.$

Recall that the graded tensor product $A\hat\otimes B$ of two $\bZ/2\bZ$-graded
algebras $A$ and $B$ is $A\otimes B$ as a vector space, but with
multiplication defined by
$$(a_1\otimes b_1)(a_2\otimes
b_2)=(-1)^{\text{deg}(b_1)\text{deg}(a_2)} a_1a_2\otimes b_1b_2.$$

\begin{lemma}
There is an isomorphism of algebras $C(V)\cong C(V_M)\hat\otimes C(V_N).$
\end{lemma}
\begin{proof}
If an orthonormal basis of $V_M$ is $\{\om_1,\dots,\om_k\}$ and an
orthonormal basis of $V_n$ is $\{\om_{k+1},\dots,\om_n\}$, then the
isomorphism is given by
$\om_{i_1}\dots\om_{i_l}\otimes\om_{j_1}\dots\om_{j_r}\mapsto
\om_{i_1}\dots\om_{i_l}\om_{j_1}\dots\om_{j_r}$, where
$i_1,\dots,i_l\in\{1,\dots,k\}$ and $j_1,\dots,j_r\in\{k+1,\dots,n\}.$
\end{proof}

Since $W_M$ acts trivially on $V_N$, and therefore $\wti W_M$ acts
trivially on every $C(V_N)$-module, we see that as $\wti W_M$-representations:
\begin{equation}\label{restspin}
\begin{aligned}
&S\cong \oplus_{2^{\dim V_N/2}} S_M,\quad \text{ if }\dim V,\dim
V_M\text{ are both even};\\
&S^\pm\cong \oplus_{2^{\dim V_N/2}} S_M^\pm,\quad \text{ if }\dim V,\dim
V_M\text{ are both odd};\\
&S^\pm\cong \oplus_{2^{(\dim V_N-1)/2}} S_M,\quad \text{ if }\dim
V\text{ is odd and }\dim V_M\text{ is even};\\
&S\cong \oplus_{2^{(\dim V_N-1)/2}} (S_M^++S_M^-),\quad \text{ if }\dim
V\text{ is even and }\dim
V_M\text{ is odd}.\\
\end{aligned}
\end{equation}

The following lemma will be our main criterion for proving that
certain induced modules have nonzero Dirac cohomology. In order to
reduce the number of cases, denote $\C S=S$ if $\dim V$ is even, and
$\C S=S^++S^-$ if $\dim V$ is odd, and similarly define $\C S_M.$ In
particular, $\C S$ is self-contragredient.

\begin{lemma}\label{indlemma} Let $\pi_M$ be an $\bH_M$-module, and $\pi=\bH\otimes_{\bH_M}\pi_M.$
\begin{enumerate}
\item[(a)] $\Hom_{\wti W}[\wti\sigma,\pi\otimes \C S]=\frac{\dim \C
    S}{\dim \C S_M}\Hom_{\wti
  W_M}[\wti\sigma|_{\wti W_M},\pi_M\otimes \C S_M].$
\item[(b)] Assume that $\pi_M$ is unitary and the $\wti W_M$-type $\wti\sigma_M$ occurs in
  the $H^D(\pi_M)$. If there exists a $\wti W$-type $\wti\sigma$ such
  that 
\begin{enumerate}
\item[(i)]$\Hom_{\wti W_M}[\wti\sigma_M,\wti\sigma]\neq 0$;
\item[(ii)] if $\Psi(\wti\sigma)=[(e,h,f)]$, then the
  central character of $\pi$ is $\chi_\pi=h$,
\end{enumerate}
then $\wti\sigma$ occurs in $H^D(\pi)$.
\end{enumerate}
\end{lemma}

\begin{proof}
Part (b) is an immediate consequence of (a) using Proposition
\ref{criterion}. To prove (a), we use Frobenius reciprocity and the
restriction of $\C S$ to $\wti W_M$:
\begin{equation*}
\begin{aligned}
\Hom_{\wti W}[\wti\sigma,\pi\otimes \C S]&=\Hom_W[\wti\sigma\otimes \C
S,\operatorname{Ind}_{W_M}^W\pi_M]=\Hom_{W_M}[(\wti\sigma\otimes \C
S)|_{W_M},\pi_M]\\
&=\Hom_{\wti W_M}[\wti\sigma|_{\wti W_M},\pi_M\otimes \C S|_{\wti W_M}]=\frac{\dim \C
    S}{\dim \C S_M}\Hom_{\wti
  W_M}[\wti\sigma|_{\wti W_M},\pi_M\otimes \C S_M].\\
\end{aligned}
\end{equation*}
\end{proof}

\subsection{Spherical modules} 
An $\bH$-module $X$ is called spherical if $\Hom_W[\triv,X]\neq 0.$
The (spherical) principal series modules of $\bH$
are defined as the induced
modules $$X(\nu)=\bH\otimes_{S(V_\bC)}\bC_\nu,$$
for $\nu\in V_\bC^\vee.$ Since $X(\nu)\cong \bC[W]$ as $W$-modules,
there is a unique {irreducible} spherical $\bH$-subquotient $L(\nu)$ of $X(\nu)$. It is
well known that:

\begin{enumerate}
\item $L(\nu)\cong L(w\nu),$ for every $w\in W$;
\item if $\nu$ is $R^+$-dominant, then $L(\nu)$ is the unique
  irreducible quotient
  of $X(\nu)$;
\item every irreducible spherical $\bH$-module is isomorphic to a
  quotient $L(\nu)$, $\nu$ $R^+$-dominant.
\end{enumerate}

Recall the Lie algebra $\fg$ that we attached to the root system
$\Phi.$ The identification $\fh=V_\bC^\vee$ allows us to view $\nu$
as an element of $\fh.$ Next, consider $\fg_1=\{x\in\fg: [\nu,x]=x$,
the $ad$ $1$-eigenspace of $\nu$ on $\fg.$ The stabilizer
$G_0=\{g\in G: Ad(g)\nu=\nu\}$ acts on $\fg_1$ with finitely many
orbits, and let $e$ be an element of the unique open dense
$G_0$-orbit. Lusztig's geometric realization of $\bH$ and
classification of irreducible $\bH$-modules implies in particular the
following statement.

\begin{proposition}\label{sphrest}
Let $\nu\in V_\bC^\vee$ be given and let $e$ be a nilpotent element of
$\fg$ attached {to $\nu$} by the procedure above. Then the spherical module
$L(\nu)$ contains the Springer representation $\sigma_{(e,1)}$ with
multiplicity one.
\end{proposition}

The second result that we need is the unitarizability of the spherical
unipotent $\bH$-modules. 

\begin{proposition}[\cite{BM}]\label{unitunip}
For every Lie triple $(e,h,f)$, the spherical module $L(\frac 12 h)$
is unitary.
\end{proposition}

Now we can state the classification of spherical modules with nonzero
Dirac cohomology.

\begin{definition}
We say that an $\bH$-module $X$ has nonzero Dirac cohomology  if for a choice of spin module $S^\ep$, $H^D_\ep(X)\neq 0$.
\end{definition}

Let $[(e,h,f)]\in\C T_0(G)$ be given and assume $G$ is simple. The
results of \cite{C} give a concrete description in every Lie type of
the map $\Psi$ from Theorem \ref{t:intro}. In particular, there is
{either} only one self-associate $\wti W$-type which we denote
$\wti\sigma_{(e,1)}$, or two associate $\wti W$-types denoted $\wti
\sigma_{(e,1)}^\pm$, which appear in the fiber $\Psi^{-1}(e,h,f)$ and
can occur in the decomposition of the tensor product
$\sigma_{(e,1)}\otimes S^\ep.$ 

\begin{corollary}\label{HDsphunip}
An irreducible spherical module $L(\nu)$ has nonzero Dirac cohomology
 if and only if $\nu=w\cdot \frac 12h$ for some
$[(e,h,f)]\in\C T_0(G).$ 
\end{corollary}

\begin{proof} Assume that $H^D_\ep(L(\nu))\neq 0.$ Then there exists a
  genuine $\wti W$-type $\wti \sigma$ occuring in
  $H^D_\ep(L(\nu))$, such that $\Psi(\wti\sigma)=[(e,h,f)]\in\C
  T_0(G)$. By Theorem \ref{t:hpv}, $\nu=w\cdot \frac 12 h$.

Conversely, fix $[(e,h,f)]\in \C T_0(G).$ The spherical module $L(\frac 12 h)$ contains $\sigma_{(e,1)}$ with multiplicity one by Proposition
\ref{sphrest}, and it is unitary by Proposition \ref{unitunip}. From
this, Proposition \ref{criterion} implies immediately that one of the
$\wti W$-types in $\Psi^{-1}(e,h,f)$ occurs in $H^D_\ep(L(\frac 12 h)),$ for some $\ep$, 
and therefore $L(\frac 12 h)$ has nonzero Dirac cohomology. 
\end{proof}

In order to investigate the precise formula for $H^D_\ep(L(h/2))$, one
uses (\ref{HDdecomp}) and the results of Borho-McPherson \cite{BMcP} about the
$W$-structure of the Springer representations in $A(e)$-isotypic
components of the full cohomology of a Springer fiber. In our setting,
this says that, as a $W$-module: 
\begin{equation}\label{bomac}
L(h/2)=\sigma_{(e,1)}+\sum_{e'>e}\sum_{\phi'\in \widehat A(e)_0}m_{e',\phi'}\sigma_{(e',\phi')},
\end{equation}
for some integers $m_{e',\phi'}\ge 0$. Here $e'>e$ means the closure
ordering of nilpotent orbits, i.e., $e\in \overline{G\cdot
  e'}\setminus G\cdot e'$. We make the following conjecture.

\begin{conjecture}\label{conj} Let $\wti\sigma$ be a $\wti W$-type such that $\Psi(\wti\sigma)=[(e',h',f')].$ Then 
$$\Hom_{\wti W}[\wti\sigma,\sigma_{(e,\phi)}\otimes S^\ep]\neq 0$$ only if $e'\ge e.$
\end{conjecture}

If this conjecture is true, then if we tensor by $S^\ep$ in (\ref{bomac}), every $\wti W$-type coming from a $\sigma_{(e',\phi')}\otimes S^\ep$, $e'>e$, would correspond under the map $\Psi$ to a triple $(e'',h'',f'')$ with $e''\ge e'>e.$ In particular, {$|h''|>|h|,$} so the formula for $D_\ep^2$ (Theorem \ref{t:dirac}) implies that none of these $\wti W$-types can contribute to $H^D_\ep(L(h/2)).$  
Thus the only nontrivial contribution to $H^D_\ep(L(h/2))$ comes from $\sigma_{(e,1)}\otimes S^\ep$, and we would have:
\begin{equation}\label{eqconj}
\begin{aligned}
&H^D_\ep(L(h/2))=[\wti\sigma_{(e,1)}:\sigma_{(e,1)}\otimes S^\ep]~\wti\sigma_{(e,1)},\text{ if }\wti\sigma_{(e,1)}\cong\wti\sigma_{(e,1)}\otimes\sgn;\\
&H^D_\ep(L(h/2))=[\wti\sigma_{(e,1)}^+:\sigma_{(e,1)}\otimes S^\ep]~\wti\sigma_{(e,1)}^++[\wti\sigma_{(e,1)}^-:\sigma_{(e,1)}\otimes S^\ep]~\wti\sigma_{(e,1)}^-,\text{ otherwise}.
\end{aligned}
\end{equation}
In section \ref{sec:3}, we will show that Conjecture \ref{conj} holds when $\bH$ is a Hecke algebra of type $A$, and therefore in that case (\ref{eqconj}) is true (see Lemma \ref{prelimresults}). Further evidence for this conjecture is provided by the computation of the Dirac index for tempered $\bH$-modules in \cite[Theorem 1]{CT}.

\section{Nonzero Dirac cohomology for type $A$}\label{sec:3}

In this section, we specialize to the case of the graded Hecke algebra
attached to the root system $\Phi=(V,R,V^\vee,R^\vee)$ of $gl(n).$ Explicitly,
$V=\bR^n$ with a basis $\{\ep_1,\dots,\ep_n\}$, $R=\{\ep_i-\ep_j: 1\le
i\neq j\le n\}.$ To simplify notation, we will also use the coordinates
$\{\ep_i\}$ to describe $V^\vee\cong \bR^n$ and $R^\vee.$ We choose
positive roots $R^+=\{\ep_i-\ep_j: 1\le i<j\le n\}$. The simple roots
are therefore $\Pi=\{\ep_i-\ep_{i+1}: 1\le i<n\}.$ The Weyl group is
the symmetric group $S_n$ and we write $s_{i,j}$ for the reflection in
the root $\ep_i-\ep_j.$

The graded Hecke algebra $\bH_n$ for $gl(n)$ is therefore generated by
$S_n$ and the {set $\{\ep_i: 1\le i\le n\}$} subject to the commutation
relations:
\begin{align*}
&s_{i,i+1} \ep_k=\ep_k s_{i,i+1},\qquad k\neq i,i+1;\\
&s_{i,i+1}\ep_i-\ep_{i+1}s_{i,i+1}=1.
\end{align*}

We review the classification of the unitary dual of $\bH_n$ and then
determine which unitary $\bH_n$-modules have nonzero Dirac cohomology.

\subsection{Langlands classification} We begin by recalling the Langlands
classification for $\bH_n$. 

\begin{definition}
The Steinberg module $\St$ is the $\bH_n$-module whose restriction to
$\bC[S_n]$ is the $\sgn$-representation, and whose only $S(V_\bC)$ weight is
$-\rho^\vee=-\frac 12\sum_{\al\in R^+} \al^\vee.$
\end{definition}

Let $\lambda=(n_1,n_2,\dots,n_r)$ be a composition of $n$. 
(This means
that $n_1+n_2+\dots+n_r=n$, but there is no order assumed between the
$n_i$'s, \eg $(2,1)$ and $(1,2)$ are different compositions of $3$.)
For every $1\le i\le r,$ regard the Hecke algebra
$\bH_{n_i}$ as the subalgebra of $\bH$ generated by
$\{\ep_j,\ep_{j+1},\dots,\ep_{j+n_i-1}\}$ and
$\{s_{j,j+1},s_{j+1,j+2},\dots, s_{j+n_i-1,j+n_i}$, where
$j=n_1+n_2+\dots+n_{i-1}+1.$ Then
$$\bH_\lambda=\bH_{n_1}\times\bH_{n_2}\times\dots\times \bH_{n_r}$$ is a (parabolic)
subalgebra of $\bH_n.$ For every $r$-tuple
$\underline\nu=(\nu_1,\nu_2,\dots,\nu_r)$ of complex numbers, we may
consider the induced module
\begin{equation}\label{indmod}
I_\lambda(\underline\nu)=\bH_n\otimes_{\bH_\lambda}(\St\otimes\bC_{\nu_1})\boxtimes\dots\boxtimes(\St\otimes\bC_{\nu_r}).
\end{equation}
If $\underline\nu$ satifies the dominance condition
\begin{equation}\label{dom}
\re(\nu_1)\ge\re(\nu_2)\ge\dots\ge\re(\nu_r),
\end{equation}
we call $I_\lambda(\underline\nu)$ a standard module.

\begin{theorem}[\cite{BZ}]\label{irrclass}
\begin{enumerate}
\item[(a)] Let $\lambda$ be a composition of $n$ and $I_\lambda(\underline\nu)$ a
  standard module as in (\ref{indmod}) and (\ref{dom}). Then
  $I_\lambda(\underline\nu)$ has a unique irreducible quotient
  $L_\lambda(\underline\nu)$.
\item[(b)] Every irreducible $\bH_n$-module is isomorphic to an
  $L_\lambda(\underline\nu)$ as in (a).
\end{enumerate}
\end{theorem}

Recall that by Young's construction, the $S_n$-types are in one to one correspondence with
partitions of $n$. We write $\sigma_\lambda$ for the $S_n$-type
parameterized by the partition $\lambda$ of $n$. For example,
$\sigma_{(n)}=\triv$ and $\sigma_{(1^n)}=\sgn.$ If $\lambda^t$ denotes
the transpose partition of $\lambda$, then $\sigma_\lambda\otimes\sgn=\sigma_{\lambda^t}.$
Finally, every composition $\lambda$ of $n$ gives rise to a partition of $n$ by reordering,
and we denote the corresponding $S_n$-type by $\sigma_\lambda$ again.

\begin{theorem}[\cite{Ro}]\label{lwt}
In the notation of Theorem \ref{irrclass}, the irreducible module
$L_\lambda(\underline\nu)$ contains the $S_n$-type $\sigma_{\lambda^t}$
with multiplicity one.
\end{theorem}

\subsection{Speh modules} The building blocks of the unitary dual are
the Speh modules whose construction we review now.

The following lemma is well-known and elementary.
\begin{lemma} For every $c\in \bC$, there exists a surjective algebra homomorphism $\tau_c:
  \bH\to \bC[S_n]$ given by:
\begin{align*}
&w\mapsto w,\quad w\in S_n;\\
&\ep_k\mapsto s_{k,k+1}+s_{k,k+2}+\dots+s_{k,n}+c,\quad 1\le k<n;\\
&\ep_n\mapsto c.
\end{align*}
\end{lemma}

\begin{proof} We check the commutation relations. It is clear that if
  $k\neq i,i+1$, $s_{i,i+1}$ commutes with $\ep_k$, since $s_{i,i+1}$
  commutes with every $s_{j,n}$, $k\le j<n$, when $i+1<k$, and it
  commutes with $s_{k,i}+s_{k,i+1}$ and $s_{k,j},$ $j\neq i,i+1$, when
  $i>k.$

Next, $s_{i,i+1}\ep_i=1+\sum_{j>i+1} w_{(i,j,i+1)}+c s_{i,i+1}$ and
$\ep_{i+1}s_{i,i+1}=\sum_{j>i+1} w_{(i,j,i+1)}+c s_{i,i+1}$, where $w_{(i,j,i+1)}$
denotes the element of $S_n$ with cycle structure $(i,j,i+1).$ The
claim follows.
\end{proof}

For every partition $\lambda$ of $n$ and $c\in \bC$, define the
$\bH$-module $\tau^*_c(\lambda)$ obtained by pulling back
$\sigma_\lambda$ to $\bH$ via $\tau_c.$

Viewing $\lambda$ as a left justified Young diagram, define the
$c$-content of the $(i,j)$ box of $\lambda$ to be $c+(j-i)$, and the
$c$-content of $\lambda$ to be the set {of} $c$-contents of boxes. This is
best explained by an example. If $\lambda$ is the partition of
$(3,3,1)$ of $n=7$, the $0$-content is the Young tableau
$$
\begin{tableau}
:.{0}.{1}.{2} \\
:.{{-1}}.{0}.{1} \\
:.{-2}\\
\end{tableau}
$$
For the $c$-content, add $c$ to the entry in every box.

\begin{lemma}\label{content}
The central character of $\tau^*_c(\lambda)$ is the ($S_n$-orbit of
the) $c$-content of the partition $\lambda$.
\end{lemma}

\begin{proof}
This follows from the known values of the simultaneous eigenvalues of
the Jucys-Murphy elements $s_{k,k+1}+s_{k,k+2}+\dots+s_{k,n}$ used to
defined $\tau_c.$ See for example \cite[Theorem 5.8]{OV}.
\end{proof}

\begin{definition}
If $\lambda$ is a box partition, i.e.,
$\lambda=(\underbrace{m,m,\dots,m}_{d})$, for some $m,d$ such that
$n=md$, and $c=0$ when $m+d$ is even or $c=\frac 12$ when $m+d$ is
odd, call the module $\tau^*_c(\lambda)$ a Speh module, and denote it
by $a(m,d).$
\end{definition}

\begin{lemma}\label{spehclass}
In the notation of Theorem \ref{irrclass}, the Speh module $a(m,d)$ is
isomorphic to $L_{\lambda^t}(\frac{m-1}2,\frac{m-3}2,\dots,-\frac{m-1}2),$
where $\lambda^t=(\underbrace{d,d,\dots,d}_{m})$.
\end{lemma}

\begin{proof}
This is immediate from Theorem \ref{irrclass}, Theorem \ref{lwt} and
Lemma \ref{content}.
\end{proof}

\subsection{The unitary dual}
The classification of irreducible $\bH_n$-modules which admit a nondegenerate
invariant hermitian form is a particular case of the classical result
of \cite{KZ}, as formulated in the Hecke algebra setting by \cite{BM}.

If $\lambda=(n_1,\dots,n_r)$ is a composition of $n$, let
$R_\lambda\subset R$
denote the root subsystem of the Levi subalgebra
$gl(n_1)\oplus\dots\oplus gl(n_r)\subset gl(n).$ If $w\in
S_n$ has the property that $wR_\lambda^+=R_\lambda^+,$ then $w$ gives
rise to an algebra automorphism of $\bH_\lambda$, and therefore $w$
acts on the set of irreducible $\bH_\lambda$-modules.

\begin{theorem}\label{herm} Let $\lambda=(n_1,\dots,n_r)$ be a composition of $n$,
  and $\underline\nu=(\nu_1,\dots,\nu_r)$ be a dominant $r$-tuple of
  complex numbers in the sense of (\ref{dom}).
In the notation of Theorem \ref{irrclass}, $L_\lambda(\underline\nu)$
is hermitian if and only if there exists $w\in S_n$ such that $w
R_\lambda^+=R_\lambda^+$
and \begin{equation}\label{e:herm}
w((\St\otimes\bC_{\nu_1})\boxtimes\dots\boxtimes(\St\otimes\bC_{\nu_r}))=(\St\otimes\bC_{-\overline\nu_1})\boxtimes\dots\boxtimes(\St\otimes\bC_{-\overline\nu_r}),
\end{equation}
as $\bH_\lambda$-modules.
\end{theorem}

\begin{corollary}
Every Speh module $a(m,d)$ is a unitary $\bH_n$-module.
\end{corollary}

\begin{proof}
Let $w_0$ denote the
longest Weyl group element in $S_n$ and $w_0(\lambda)$ the longest
Weyl group element in $S_{n_1}\times\dots\times S_{n_r}.$
 Using Lemma \ref{spehclass}, we see now that every Speh
module $a(m,d)$ is hermitian since the Weyl group element $w_0 w_0(\lambda^t)$
satifies condition (\ref{e:herm}) in this case. 

Since in addition $a(m,d)$ is irreducible as an
$S_n$-module, it is in fact unitary.
\end{proof}

The classification of the unitary dual of $\bH_n$ is also well-known
(see \cite{T} for the classification of the unitary dual for $GL(n,\bQ_p)$).

The building blocks are the Speh modules defined before. First,
every Speh module $a(m,d)$ can be tensored with a unitary character
$\bC_{y}$, $y\in\sqrt{-1}\bR$ by which the central element
$\ep_1+\dots+\ep_n$ of $\bH_n$ acts. We denote the resulting (unitary)
irreducible module by $a_y(m,d).$ 

Next, we consider induced complementary series representations of the form 
\begin{equation}\label{deform}
\pi(a_y(m,d),\nu)=\bH_{2k}\otimes_{\bH_k\times\bH_k}(a_y(m,d)\otimes
\bC_\nu)\boxtimes (a_y(m,d)\otimes\bC_{-\nu}),\quad 0<\nu<\frac 12;
\end{equation}
in this notation, it is implicit that $k=md$.
An easy deformation argument shows that all $\pi(a_y(m,d))$ are
irreducible unitary $\bH_{2k}$-modules.

\begin{theorem}[\cite{T}]\label{unitdual} 

\begin{enumerate}
\item[(a)] Let $\lambda=(n_1,\dots,n_r)$ be a composition of $n$. If every
  $\pi_1,\dots,\pi_r$ is either a Speh module of the form $a_y(m,d)$
  or an induced complementary series module of the form
  $\pi(a_y(m,d),\nu)$ as in (\ref{deform}), then the induced module
\begin{equation}\label{tadicform}
\bH\otimes_{\bH_\lambda}(\pi_1\boxtimes\dots\boxtimes\pi_r)
\end{equation}
is irreducible and unitary. Moreover, two such modules are isomorphic
if and only if one is obtained from the other one by permuting the
factors.
\item[(b)] Every unitary $\bH_n$-module is of the form (\ref{tadicform}).
\end{enumerate}

\end{theorem}

\subsection{Nilpotent orbits in $sl(n)$}
The classification of nilpotent orbits for $sl(n)$ is well-known. Let
$P(n)$ denote the set of all (decreasing) partitions of $n$ and let $DP(n)$ be the set of
 partitions with
 distinct sizes. The
Jordan canonical form gives a bijection between the set of nilpotent
orbits of $sl(n)$ and $P(n)$. If $(e_\lambda,h_\lambda,f_\lambda)$ is a Lie triple, where the
nilpotent element $e_\lambda$ is the Jordan form given by the
partition $\lambda=(n_1,n_2,\dots,n_r),$ $n_1\ge n_2\ge\dots\ge
n_r>0$, then, using the
identification $\fh=\bC^n$, the middle element $h_\lambda$ can be
chosen to have coordinates
\begin{equation}\label{hlam}
h_\lambda=\left(\frac{n_1-1}2,\dots,-\frac{n_1-1}2;\dots;\frac{n_r-1}2,\dots,-\frac{n_r-1}2\right).
\end{equation}
If we write $\lambda$ as
$\lambda=(\underbrace{n_1',\dots,n_1'}_{k_1},\underbrace{n_2',\dots,n_2'}_{k_2},\dots,\underbrace{n_l',\dots,n_l'}_{k_l}),$
with $n_1'>n_2'>\dots>n_l'>0$, then the centralizer in $gl(n)$ of the
triple $(e_\lambda,h_\lambda,f_\lambda)$ is $gl(k_1)\oplus
gl(k_2)\oplus\dots\oplus gl(k_l).$ In particular, the centralizer in $sl(n)$ is a
toral subalgebra if and only if $\lambda\in DP(n)$. Thus, we have a
natural bijection $\C T_0(SL(n))\leftrightarrow DP(n).$
For $\lambda\in P(n)$, ($\C T_0$ defined in (\ref{eq:tzero}) )
viewed as a left justified Young tableau, define
\begin{equation}\label{hooks}
\text{hook}(\lambda)
\end{equation}
to be the partition obtained by taking the hooks of $\lambda.$ For
example, if $\lambda=(3,3,1),$ then $\text{hook}(\lambda)=(5,2).$ It
is clear that $\text{hook}(\lambda)\in DP(n).$
 
We will need the following reformulation for the central character of
a Speh module.
\begin{lemma}\label{charspeh}
The central character of a Speh module $a(m,d)$ is the ($S_n$-orbit
of) $h_{\lambda'}$ (see (\ref{hlam})), where $\lambda'$ is the
partition
$$\lambda'=\text{hook}(\underbrace{m,m,\dots,m}_d)=(m+d-1,m+d-3,\dots, |m-d|+1).$$
\end{lemma}

\begin{proof}
This is immediate from Lemma \ref{content} and (\ref{hlam}).
\end{proof}
\subsection{Irreducible $\wti S_n$-representations}
Denote the length of a partition $\lambda$ by $|\lambda|$. We say that $\lambda$ is even (resp. odd) if
 $n-|\lambda|$ is even (resp. odd). 
The first part of Theorem \ref{t:intro} for $\wti S_n$ is a classical result of Schur.

\begin{theorem}[Schur, \cite{St}] The irreducible $\wti
  S_n$-representations are parameterized by partitions in 
$DP(n)$ as follows:
\begin{enumerate}
\item[(i)] for every
  even $\lambda\in DP(n)$, there exists a unique
  $\wti\sigma_\lambda\in \widehat{\wti S_n}$;
\item[(ii)] for every odd
  $\lambda\in DP(n)$, there exist two associate $
  \wti\sigma_\lambda^+,\wti\sigma_\lambda^-\in \widehat{\wti
  S_n}$. 
\end{enumerate}
The dimension of $\wti\sigma_\lambda$
  or $\wti\sigma_\lambda^\pm$, where
  $\lambda=(\lambda_1,\dots,\lambda_m)\in DP(n)$, is
\begin{equation}
2^{[\frac{n-m}2]}\frac{n!}{\lambda_1!\dots\lambda_m!}\prod_{1\le
  i<j\le m}\frac{\lambda_i-\lambda_j}{\lambda_i+\lambda_j}.
\end{equation}
\end{theorem}

In order to simplify the
formulas below, we let
$\wti\sigma_\lambda^\ep$ denote any one of $\wti\sigma_\lambda$, if $\lambda$ is an even partition in $DP(n)$, or $\wti\sigma_\lambda^\pm$,
if $\lambda$ is an odd partition in $DP(n)$.

The decomposition of the tensor product of an $S_n$-type
$\sigma_\mu$ with a spin representation $\wti\sigma_{(n)}$ is
known. 
\begin{theorem}[{\cite[Theorem 9.3]{St},\cite[Chapter 3, (8.17)]{Mac}}]\label{glammu}
If $\lambda\neq (n)$, we have:
\begin{equation}\label{tensdecomp}
\dim\Hom_{\wti
  S_n}[\wti\sigma_\lambda,\sigma_\mu\otimes \wti\sigma_{(n)}]=\frac
1{\epsilon_\lambda\epsilon_{(n)}} 2^{\frac{|\lambda|-1}2} g_{\lambda,\mu},
\end{equation}
where $\ep_\lambda=1$ (resp. $\ep_\lambda=\sqrt 2$) if $\lambda$ is
even (resp. odd), and the integer $g_{\lambda,\mu}$ is the $(\lambda,\mu)$ entry in the inverse matrix $K(-1)^{-1}$, where $K(t)$ is the matrix of Kostka-Foulkes polynomials. In particular:
\begin{enumerate}
\item[(i)] $g_{\lambda,\lambda}=1$;
\item[(ii)] $g_{\lambda,\mu}=0$, unless $\lambda\ge\mu$ in the ordering of partitions. 
\end{enumerate}
\end{theorem}

\begin{example}\label{tensorhook}
The integers $g_{\lambda,\mu}$ have also an explicit combinatorial description in terms of ``shifted tableaux'' of unshifted shape $\mu$ and content $\lambda$ satifying certain admissibility conditions (see \cite[Theorem 9.3]{St}). From this description, one may see for example that if $\lambda=\text{hook}(\mu)$, then $g_{\lambda,\mu}=1$ in (\ref{tensdecomp}).
\end{example}

\subsection{Nonzero cohomology} We are now in position to determine
the unitary modules of $\bH_n$ with nonzero Dirac cohomology.

We remark that since $gl(n)$ is not semisimple, the spin modules $S^\ep$ of
$C(V)$ ($V\cong \bC^n$) are not necessarily irreducible $\wti
S_n$-representations. More precisely, using (\ref{restspin}), we see that
$S^\pm|_{\wti S_n}=\wti\sigma_{(n)},$ when $n$ is odd, and $S|_{\wti S_n} =\wti\sigma_{(n)}^++\wti\sigma_{(n)}^-$, when $n$ is even. 

\begin{lemma}\label{ccdirac}
Assume $X$ is an irreducible $\bH_n$-module such that $H^D(X)\neq 0.$
Then the central character of $X$ is in the set $\{h_{\lambda}/2:
\lambda\in DP(n)\},$ where $h_\lambda$ is as in (\ref{hlam}).
\end{lemma}

\begin{proof}
This is just a reformulation of Theorem \ref{t:hpv} in this particular case.
\end{proof}

As a consequence of (\ref{tensdecomp}), we obtain the following precise results for Dirac cohomology.

\begin{lemma}\label{prelimresults}
\item[(a)] A spherical module $L(\nu)$ has nonzero Dirac cohomology if and only if $\nu\in \{h_{\lambda}/2: 
\lambda\in DP(n)\},$ where $h_\lambda$ is as in (\ref{hlam}), and in this case $H^D_\ep(L(h_{(n)}/2))=S^\ep$, and if $\lambda\neq (n)$:
\begin{align*}
H^D_\ep(L(h_\lambda/2))&=2^{[(|\lambda|-1)/2]}\wti\sigma_\lambda, &\text{ if $n$ is odd and $\lambda$ is even};\\
&=2^{[(|\lambda|-1)/2]}(\wti\sigma_\lambda^++\wti\sigma_\lambda^-), &\text{ if $n$ is odd and $\lambda$ is odd};\\
&=2^{[(|\lambda|)/2-1]}(\wti\sigma_\lambda^\ep+\wti\sigma_\lambda^\ep\otimes\sgn), &\text{ if $n$ is even}.\\
\end{align*}
\item[(b)] Every Speh module $a(m,d)$ has nonzero Dirac cohomology. More precisely, 
$H^D_\ep(a(m,d))=2^{(d-1)/2}~(\wti\sigma_{(m+d-1,m+d-3,\dots,|m-d|+1)}^++\wti\sigma_{(m+d-1,m+d-3,\dots,|m-d|+1)}^-),$ if $d$  is odd and $m$ is even, $H^D_\ep(a(m,d))=2^{[(d-1)/2]}~\wti\sigma_{(m+d-1,m+d-3,\dots,|m-d|+1)}^\ep,$ if $d$ is odd and $m$ is odd,
or $H^D_\ep(a(m,d))=2^{[(d+1)/2]}~\wti\sigma_{(m+d-1,m+d-3,\dots,|m-d|+1)}^\ep,$ otherwise.
\item[(c)] Every complementary series induced module $\pi(a_y(m,d),\nu)$ as
  in (\ref{deform}) has zero Dirac cohomology.
\end{lemma}

\begin{proof}
(a) This is immediate by (\ref{bomac}) and the upper unitriangular property of the numbers $g_{\lambda,\mu}$ in Theorem \ref{glammu}.

(b) By Lemma \ref{charspeh}, the central character of $a(m,d)$ is
$h_{\lambda'}$ where $\lambda'=\text{hook}(\lambda)\in DP(n).$ By
Example \ref{tensorhook}, the genuine $\wti S_n$-type
$\wti\sigma_{\lambda'}$ occurs with nonzero multiplicity in
$\sigma_{(m,m,\dots,m)}\otimes S.$ By construction, $a(m,d)$ is
isomorphic with $\sigma_{(m,m,\dots,m)}$ as
$S_n$-representations. This means that the hypothesis of Proposition
\ref{criterion} are satisfied, hence $\wti\sigma_{\lambda'}$ occurs in $H^D(a(m,d))$.

(c) This is immediate from Lemma \ref{ccdirac}, since $a_y(m,d)$,
$y\neq 0$ and $\pi(a_y(m,d),\nu)$, $0<\nu<\frac 12$ do not have the allowable
central characters.
\end{proof}

\begin{theorem}\label{main}
An irreducible unitary $\bH_n$-module has nonzero Dirac cohomology if
and only if it is isomorphic with an induced module
\begin{equation}\label{possible}
\begin{aligned}
&X=\bH_n\otimes_{\bH_{ev}\otimes
  \bH_{odd}}(\pi_{ev}\boxtimes\pi_{odd}),\quad \text{where}\\
&\bH_{ev}=\bH_{k_1}\times \bH_{k_2}\times\dots\times \bH_{k_\ell},\
\bH_{odd}=\bH_{k'_1}\times \bH_{k'_2}\times\dots\times
\bH_{k'_t},\quad \text{and}\\
&\pi_{ev}=a(m_1,d_1)\boxtimes a(m_2,d_2)\boxtimes \dots\boxtimes
  a(m_\ell,d_\ell),\ \pi_{odd}=a(m_1',d_1')\boxtimes a(m_2',d_2')\boxtimes \dots\boxtimes
  a(m_t',d_t'),\\
&m_i+d_i\equiv 0 (\text{mod }2),\ m_j'+d_j'\equiv 1 (\text{mod }2),
\end{aligned}
\end{equation}
$k_1+k_2+\dots+k_\ell+k_1'+k_2'+\dots+k_t'=n$,
where $a(m_i,d_i), a(m_j',d_j')$ are Speh modules for $\bH_{k_i},\bH_{k'_j}$ and such that the
following conditions are satisfied:
\begin{equation}\label{gaps}
\begin{aligned}
&m_1+d_1-1\ge |m_1-d_1|+1>m_2+d_2-1\ge
|m_2-d_2|+1>\dots>m_\ell+d_\ell-1;\\
&m_1'+d_1'-1\ge |m_1'-d_1'|+1>m_2'+d_2'-1\ge
|m_2'-d_2'|+1>\dots>m_t'+d_t'-1.\\
\end{aligned}
\end{equation}
\end{theorem}

\begin{proof} From Theorem \ref{unitdual}, a unitary irreducible
  module $X$ is induced from a combination of Speh modules and
  complementary series modules. It is immediate that in order for $X$
  to have one of the central characters from Lemma \ref{ccdirac}, a
  first restriction is that only Speh modules can appear in the
  induction, so $X$ is of the form (\ref{possible}). Notice then that
  the central character of $X$ is obtained by concatenating the
  central characters of $a(m_i,d_i).$ Therefore the central character
  of $X$ is $S_n$-conjugate to $h_\lambda$, where $\lambda$ is the
  composition $\lambda=\lambda^1\sqcup\dots\sqcup\lambda^\ell\sqcup\mu^1\sqcup\dots\sqcup\mu^t$, where $\lambda^i=(m_i+d_i-1,m_i+d_i-3,\dots, |m_i-d_i|+1)$, $1\le i\le \ell$ and $\mu^j=(m_j'+d_j'-1,m_j'+d_j'-3,\dots, |m_j'-d_j'|+1)$, $1\le j\le t$. The entries in the first type of strings are all even, while
the entries in the second type of strings are all odd. Since we need $\lambda$
to have no repetitions, condition (\ref{gaps}) follows.

{For the converse, assume $X$ is as in (\ref{possible}) and (\ref{gaps}). Then the
central character of $X$ is $h_\lambda$, where $\lambda$ is as
above. By Proposition \ref{criterion}, it remains to check that
$X\otimes S$ contains the $\wti S_n$-type $\wti \sigma_\lambda$. 
From Lemma \ref{indlemma}, we see that $\Hom_{\wti
  S_n}[\wti\sigma_\lambda,X\otimes S]=\frac{\dim\C S}{\dim \C
  S_M}\Hom_{\wti W_{M}}[\wti\sigma|_{\wti
  W_M},(\pi_{ev}\boxtimes\pi_{odd})|_{W_M}\otimes \C S_M],$ where
$\wti W_M=\wti S_{k_1}\cdot \dotsc\cdot \wti S_{k_\ell}\cdot \wti
S_{k_1'}\cdot\dotsc\cdot \wti S_{k'_t}$, and $\C S_M$ is the
corresponding spin module. (Here $\cdot$ denotes the graded version of
the direct product coming from the graded tensor product of Clifford
algebras as in Section \ref{s:2.6}.) From Lemma \ref{prelimresults},
we know that the $\wti S_{k_i}$-representation
$\wti\sigma_{\lambda^i}$ occurs in $a(m_i,d_i)|_{S_{k_i}}$ tensored with
the spin $\wti S_{k_i}$-module and similarly the $\wti
S_{k'_j}$-representation $\wti\sigma_{\mu^j}$ occurs in
$a(m'_j,d'_j)|_{S_{k'_j}}$ tensored with the spin $\wti
S_{k'_j}$-module. Therefore the tensor product representation
$\wti\sigma_{\lambda,M}:=\wti
\sigma_{\lambda^1}\boxtimes\dots\boxtimes\wti\sigma_{\lambda^\ell}\boxtimes
\wti \sigma_{\mu^1}\boxtimes\dots\boxtimes\wti\sigma_{\mu^t}$ occurs
in $(\pi_{ev}\boxtimes\pi_{odd})|_{W_M}\otimes \C S_M$. Finally, since
the composition $\lambda$ is just the concatenation of the
$(\lambda^i)$'s and the $(\mu^j)$'s, one sees that
$\wti\sigma_{\lambda,M}$ occurs with nonzero multiplicity in
$\wti\sigma_\lambda|_{\wti W_M}$. 

}

\end{proof}


\ifx\undefined\bysame
\newcommand{\bysame}{\leavevmode\hbox to3em{\hrulefill}\,}
\fi

\end{document}